\documentclass[12]{amsart}
\usepackage{mathrsfs}
\usepackage{stmaryrd}
\usepackage[latin1]{inputenc}
\usepackage{amsmath}
\usepackage{amsfonts}
\usepackage{amssymb}
\usepackage[all,cmtip]{xy}
\usepackage{verbatim}   % This is for the \begin{comment} environment.
\usepackage{amsthm}
\usepackage{amscd}

\usepackage{color}
\definecolor{ghcolor}{RGB}{0, 150, 200} %defines colors if you need in hyperref
\definecolor{winestain}{rgb}{0.5,0,0}
\usepackage[linktocpage=true, hidelinks, colorlinks, linkcolor=blue, citecolor=red]{hyperref} %linktocpage makes pages in toc clickable

\theoremstyle{plain}
\newtheorem{thm}{Theorem}[section] %here!! Section or subsection, about the counter!!!
\newtheorem{defn}[thm]{Definition}
\newtheorem{prop}[thm]{Proposition}

\newtheorem{lemma}[thm]{Lemma}
\newtheorem{corollary}[thm]{Corollary}

\theoremstyle{definition}
\newtheorem{rem}[thm]{Remark}

\newcommand{\Zp}{\mathbb{Z}_p}
\newcommand{\Qp}{\mathbb{Q}_p}

\newcommand{\Hom}{\textnormal{Hom}}

\newcommand{\Fil}{\textnormal{Fil}}
\newcommand{\Mod}{\textnormal{Mod}}

\newcommand{\Gal}{\textnormal{Gal}}
\newcommand{\Ker}{\textnormal{Ker}}
\newcommand{\Coker}{\textnormal{Coker}}

\newcommand{\Mat}{\textnormal{Mat}}

\newcommand{\M}{\mathcal{M}}

\newcommand{\bolda}{\boldsymbol{\alpha}}
\newcommand{\boldb}{\boldsymbol{\beta}}
\newcommand{\bolde}{\boldsymbol{e}}

\newcommand{\FM}{\Fil^r \M}
\newcommand{\fr}{\phi_r}

\newcommand{\Modus}{\Mod_{k[u]/u^s}^{\phi}}
\newcommand{\Modut}{\Mod_{k[u]/u^t}^{\phi}}

\newcommand{\alphaseq}{\alpha_1, \cdots, \alpha_d}
\newcommand{\eseq}{e_1, \cdots, e_d}

\title{A NOTE ON TORSION BREUIL MODULES IN THE CASE $er=p-1$}
\author{HUI GAO}
\address{Beijing International Center for Mathematical Research, Peking University, No. 5 Yiheyuan Road, Haidian District, Beijing 100871, China}
\email{gaohui@math.pku.edu.cn}
\begin{document}
\maketitle

\pagestyle{myheadings}
\markright{A NOTE ON TORSION BREUIL MODULES IN THE CASE $er=p-1$}

\begin{abstract}
In this note, we prove that the category of unipotent torsion Breuil modules is an abelian category, under the condition $er=p-1, r<p-1$.
%As a side product, we point out that some of the approaches in Caruso's work on torsion log-syntomic cohomologies cannot be generalized to $er=p-1$.
\end{abstract}
\tableofcontents

\section*{Introduction} \label{section intro}

Let $p$ be a prime, $k$ a perfect field of characteristic $p$, $W(k)$ the ring of Witt vectors, $K_0 = W(k)[\frac{1}{p}]$ the fraction field, $K$ a finite totally ramified extension of $K_0$, $e =e(K/K_0)$ the ramification index, $\overline K$ a fixed algebraic closure, and $G_K =\Gal(\overline{K}/K)$ the absolute Galois group. We will use $r \in \mathbb Z^{+}$ to denote an integer such that $er \leq p-1$. Many of the results in this note are valid for $er \leq p-1$, but we will be most interested in the case $er=p-1$.

Integral $p$-adic Hodge theory is about the study of integral lattices (that is, $\Zp$-lattices) in semi-stable Galois representations. A feature in integral $p$-adic Hodge theory is to use linear algebra data to study representations, and Breuil modules is one of the most useful linear algebra tool. Compared with the usual $p$-adic Hodge theory (that is, the study of $\Qp$-linear representations), one of the advantages of integral $p$-adic Hodge theory is that we can take reductions to study torsion phenomenon. In this note, we study torsion Breuil modules, and we prove that the category of \emph{unipotent} torsion Breuil modules is an abelian category when $er=p-1$ and $r<p-1$ (Let us mention here that the term ``unipotent" is a technical linear algebra condition).

Our result is known when $er \leq p-2$ by \cite{Car06}.
The main motivation for considering the categories in the current paper is to see if we can generalize the work of \cite{Car08} on torsion log-syntomic cohomologies to the case $er=p-1$.
However, during the work, we found out that the original approach of \textit{loc. cit.} cannot be generalized, so it is clear that we need some new ideas, and we hope we can return to the question in the future work. \footnote{In an earlier preprint version of this paper, arXiv:1410.3953, we have given a sketch of the reason why the work of \cite{Car08} cannot be generalized, see Section 3 of the arXiv version.}
As a final note, we hope that these results will be useful for studying (crystalline) Galois representations in the future.
In particular, these results should be useful for studying reductions of crystalline representations, which have potential applications to study deformation rings (and modularity lifting theorems).

Let us give a sketch of the strategy of the paper.
We first show that the subcategory consisting of $p$-torsion unipotent Breuil modules is abelian by proving that it is equivalent to another abelian category $\Mod_{k[u]/u^p}^{\phi, u}$ (see Section \ref{section mod with fil}).
Let us point out here that although it is relatively straightforward to define what it means for a $p$-torsion Breuil module to be unipotent, it is not so straightforward for a general $p^n$-torsion Breuil module; rather, we have to define it in an inductive manner.
Once we obtain the correct definition of unipotency, a d\'evissage argument will then prove our main result.

\textbf{Acknowledgement}: The author is very grateful to Tong Liu for initially suggesting these questions, and many useful discussions. This note is written when the author is a postdoc in Beijing International Center for Mathematical Research. The author would like to thank the institute for the hospitality, and Ruochuan Liu for being his postdoc mentor. This work is partially supported by China Postdoctoral Science Foundation General Financial Grant 2014M550539.

\textbf{Notations}: We use $[a_1, \cdots, a_d]$ to denote a diagonal matrix with the elements in the bracket. We use $A^T$ to denote the transpose of a matrix. In particular, $(e_1, \cdots, e_d)^T$ is a column vector. We use a boldface letter to mean a column vector, e.g., $\bolde$ or $\bolda$. We use notations like $\oplus R \bolde$ to denote the space of $R$-span of vectors in $\bolde$, e.g., if $\bolde=(e_1, \cdots, e_d)^T$, i.e., $\bolde$ is a column vector, then $\oplus R \bolde =\oplus_{i=1}^d Re_i$. When the ring $R$ is clear from the context, we can omit it and simply denote $\oplus \bolde$.

\section{Modules with filtrations} \label{section mod with fil}

In this section, we define certain categories of modules with filtrations, and prove that they are abelian categories under some conditions.

Let $\pi$ be a fixed uniformizer of $K$. Let $E(u) \in W(k)[u]$ be the minimal polynomial of $\pi$ over $K_0$, which is of degree $e$.
Recall that $S$ is the $p$-adic completion of the PD-envelope of $W(k)[u]$ with respect to the ideal $(E(u))$, which is a $W(k)[u]$-subalgebra of $K_{0}[\![u]\!]$, and $S = \{\sum_{i=0}^{\infty} a_i \frac{E(u)^i}{i!} | a_i \in W(k)[u], a_i \to 0 \ p-\text{adically}  \}$.
$S$ has a filtration $\{\Fil^{i}S\}_{i \geq 0}$, where $\Fil^{i}S$ is the $p$-adic completion of the ideal generated by all $\gamma_{j}(E(u))=\frac{E(u)^{j}}{j!}$ with $j\geq i$.
There is a Frobenius $\phi: S \to S$ which acts on $W(k)$ via Frobenius and sends $u$ to $u^p$, and a $W(k)$-linear differentiation $N$ (called the monodromy operator) such that $N(u) = -u$.
We denote $c =\frac{\phi(E(u))}{p}$, which is a unit in $S$. we also denote $S_n= S/p^nS$. Note that $\phi(\Fil^i S) \subseteq p^iS$ for $1 \leq i \leq p-1$, and we denote $\phi_i :=\frac{\phi}{p^i}: \Fil^i S \to S$ for $1 \leq i \leq p-1$.

Let $S_1=S/pS$, $\Fil^r S_1: =\Fil^r S+pS/pS \simeq \Fil^rS/p\Fil^rS$, and $\phi_r: \Fil^r S_1 \to S_1$ the map induced from $\phi_r: \Fil^r S \to S$.

Let $T_{ep}:=S_1/\Fil^pS_1 \simeq k[u]/u^{ep}$, $\Fil^r T_{ep}:=u^{er}\cdot k[u]/u^{ep}$, $\phi_r: \Fil^r T_{ep} \to T_{ep}$ with $\phi_r(u^{er}):= (\frac{\phi(E(u))-E(u)^p}{p} )^r (\bmod p)$.
Note that if $E(u)=u^e+pF(u)$, then $\frac{\phi(E(u))-E(u)^p}{p} (\bmod p) = \overline{\phi(F(u))}$, which is a unit in $T_{ep}$ because $F(0)$ is a unit in $W(k)$.

When $er \leq p-1$, $p \leq s \leq ep$, we define $T_s:=k[u]/u^s$, and let $\Fil^r T_s: =u^{er}\cdot k[u]/u^s$, and $\phi_r: \Fil^r T_s \to T_s$ such that $\phi_r(u^{er})=(\overline{\phi(F(u))}^r$.

When $er \leq p-1$ and $p \leq s \leq ep$, it can be easily checked that the following diagram is commutative,

$$\begin{CD}
\Fil^r S @>>> \Fil^r S_1  @>>>  (u^{er})/u^{ep} @>>>(u^{er})/u^s \\
@V\phi_rVV     @V\phi_rVV     @V\phi_rVV      @V\phi_rVV    \\
S@>>>    S_1 @>>>   T_{ep}=k[u]/u^{ep} @>>> T_s=k[u]/u^s.
\end{CD}$$
In the following, we will just denote $\overline{\phi(F(u))}$ by $c$ for simplicity.

\begin{defn}
For $er\leq p-1, p \leq s \leq ep$, the category $\Mod_{k[u]/u^s}^{\phi}$ consists of objects $(\M, \FM, \phi_r)$ where
\begin{itemize}
\item $\M$ is a finite free $T_s$-module.
\item $\FM$ is a submodule of $\M$ containing $u^{er}\M$.
\item $\phi_r: \FM \to \M$ is a map such that $\phi_r(ax)=\phi(a)\phi_r(x)$ for any $a \in T_s, x \in \FM$, and the image generates $\M$.
\end{itemize}
Morphisms in the category are $T_s$-homomorphisms that are compatible with filtrations and $\fr$.
\end{defn}

A sequence $0 \to \M_1 \to \M \to \M_2 \to 0$ in $\Mod_{k[u]/u^s}^{\phi}$ is called short exact if it is short exact as $T_s$-modules, and the sequence on filtrations $0 \to \FM_1 \to \FM \to \FM_2 \to 0$ is also short exact. In this case, we call $\M_2$ a quotient of $\M$.

The Cartier dual of $\M$ is defined by $\M^{\vee}:= \Hom_{T_s}(\M, T_s)$,
$$\Fil^r \M^{\vee} := \{ f \in \M^{\vee}, f(\Fil^r \M) \subseteq \Fil^r T_s \},$$
and
$$\varphi_r^{\vee} : \Fil^r \M^{\vee} \to \M^{\vee}, \text{for all } x \in \Fil^r \M, \varphi_r^{\vee}(f)(\varphi_r(x)) = \varphi_r(f(x)).$$
Note that $\varphi_r^{\vee}(f)$ is well defined since $\varphi_r(\Fil^r \M)$ generates $\M$.

\begin{lemma}\label{submodule}
\begin{enumerate}

\item Given a $k[u]/u^s$-module $\M$ of finite type, suppose it can be written as a direct sum $\M =\oplus_{i=1}^d T_s m_i$ where $m_i \neq 0$ ($\M$ is not necessarily free). Then for any submodule $\mathcal N$ of $\M$, we can choose some nonzero elements $e_i \in \M, 1 \leq i \leq d$ such that $\M =\oplus_{i=1}^d T_s e_i$ and $\mathcal N= \oplus_{i=1}^{d}T_s u^{x_i}e_i$ where $0 \leq x_i <s$. Note that here we allow $u^{x_i}e_i$ to be $0$.

\item Given a $k[u]/u^s$-module $\M$ of finite type, if $\M=\oplus_{i=1}^a \alpha_i = \oplus_{j=1}^b \beta_j$ where $\alpha_i, \beta_j \neq 0, \forall i, j$, then $a=b$.
\end{enumerate}
\end{lemma}
\begin{proof}
Statement (1) can be proved similarly as Lemma 3.2.1 of \cite{Car06}. Statement (2) is easily deduced by an induction on $\textnormal{min} \{a, b\}$, using similar idea as the proof of (1).
\end{proof}

\begin{lemma}
The Cartier dual functor induces an anti-equivalence (thus a duality), and it transforms short exact sequences to short exact sequences.
\end{lemma}
\begin{proof}
By Lemma \ref{submodule}, given any $\M \in \Mod_{k[u]/u^s}^{\phi}$, there is a ``base adapt\'{e}e" for $\Fil^r \M$. Then this proposition is similarly proved as Proposition V.3.1.6 of \cite{Car05}.
\end{proof}

For $\M \in \Modus$, let $(\phi^{\ast})^1\M=\phi^{\ast}\M$ be the $T_s$-span of $\phi_r(u^{er}\M)$, and inductively, define$(\phi^{\ast})^n\M$ as the $T_s$-span of $\phi_r(u^{er}(\phi^{\ast})^{n-1}\M)$. Let $\M^{\textnormal m}:= \cap_{n=1}^{\infty} (\phi^{\ast})^n \M$, and $\Fil^r \M^{\textnormal m} :=\M^{\textnormal m} \cap \Fil^r \M$.

\begin{prop} \label{prop max mult}
$\M^{\textnormal m}$ is a finite free $T_s$-module, and $\Fil^r \M^{\textnormal m}=u^{er}\M^{\textnormal m}$. Indeed, with the induced $\phi_r$-structure, $\M^{\textnormal m} \in \Modus$.
\end{prop}

\begin{proof}
$\M^{\textnormal m}$ is a submodule of $\M$, by Lemma \ref{submodule}, we can choose a basis $\eseq$ of $\M$, such that $\M^{\textnormal m}= \oplus_{i=1}^{a} T_s u^{x_i}e_i $, where $a
\leq d$, $0 \leq x_i <s$.
We claim that $x_i =0$ for all $i$. To prove the claim, suppose otherwise, then without loss of generality, we can assume $x_1 >0$ and is maximal among all $x_i$.
By the definition of $\M^{\textnormal m}$, $\phi_r(u^{er} \M^{\textnormal m})$ generates $\M^{\textnormal m}$, it means that
$(\phi_r(u^{er}u^{x_i}e_1), \cdots,\phi_r(u^{er}u^{x_a}e_a))^T = C (u^{x_i}e_1, \cdots, u^{x_a}e_a )^T$ for some invertible $a \times a$-matrix $C$. Now

\begin{eqnarray*}
(\phi_r(u^{er}u^{x_i}e_1), \cdots,\phi_r(u^{er}u^{x_a}e_a))^T &=& (u^{px_1}\phi_r(u^{er}e_1), \cdots, u^{px_a}\phi_r(u^{er}e_a))^T\\
&=& [u^{px_1}, \cdots, u^{px_a}] (A, B) (e_1, \cdots, e_a, \cdots, e_d)^T,
\end{eqnarray*}
where $A$ is an $a\times a$-matrix and $B$ is an $a\times (d-a)$-matrix.

Thus, we have $[u^{px_1}, \cdots, u^{px_a}]A =C [u^{x_1}, \cdots, u^{x_a}]$. With this equality, we can easily show that all elements in the first row of $C$ are divisible by $u$ (using that $x_1$ is nonzero and maximal among all $x_i$), which contradicts that $C$ is invertible! Thus our claim is proved, and $\M^{\textnormal m}$ is finite free.

Since we have already shown that $\M^{\textnormal m}$ is finite free, so we can choose a basis $(e_1, \cdots, e_a)$ of $\M^{\textnormal m}$ such that $\Fil^r \M^{\textnormal m}= \oplus_{i=1}^{a} T_s u^{x_i}e_i $ for some $0 \leq x_i \leq s$. Since clearly $\Fil^r \M^{\textnormal m} \supseteq u^{er}\M^{\textnormal m}$, we must have $0 \leq x_i \leq er$. We claim that $x_i=er$ for all $i$. To prove the claim, suppose otherwise, and we can assume $x_1<er$. Note that $u^{x_i}e_i \in \Fil^r \M$, so $\phi_r(u^{x_i}e_i) \in \M$ (not necessarily in $\M^{\textnormal m}$). Again we use the fact that $\phi_r(u^{er} \M^{\textnormal m})$ generates $\M^{\textnormal m}$. So

\begin{eqnarray*}
(\phi_r(u^{er}e_1), \cdots,\phi_r(u^{er}e_a))^T &=&
[u^{p(er-x_1)}, \cdots,u^{p(er-x_a)}] (\phi_r(u^{x_1}e_1), \cdots,\phi_r(u^{x_a}e_a))^T\\
&=& [u^{p(er-x_1)}, \cdots,u^{p(er-x_a)}] (A, B)(e_1, \cdots, e_a, \cdots, e_d)^T,
\end{eqnarray*}
where $A$ ia an $a\times a$-matrix and $B$ is an $a\times (d-a)$-matrix. Thus $[u^{p(er-x_1)}, \cdots,u^{p(er-x_a)}] A=C$ for some invertible $a\times a$-matrix $C$. But this is impossible, because all elements on the first row of $C$ will be divisible by $u$. So we have finished our proof.
\end{proof}

\begin{defn}
$\M \in \Modus$ is called multiplicative if $\Fil^r \M =u^{er}\M$, it is called \'{e}tale if $\Fil^r \M =\M$. It is called nilpotent if it has no nonzero multiplicative submodules, it is called unipotent if it has no nonzero \'{e}tale quotients.

We will use $\Mod_{k[u]/u^s}^{\phi,u}$ to denote the subcategory consisting of unipotent objects.
\end{defn}

\begin{prop}
For $\M \in \Mod_{k[u]/u^s}^{\phi}$,
\begin{enumerate}
\item We have short exact sequences
$$0 \to \M^{\textnormal{m}} \to \M \to \M^{\textnormal{nil}} \to 0$$
and
$$0 \to \M^{\textnormal{uni}} \to \M \to \M^{\textnormal{et}} \to 0,$$
where $\M^{\textnormal{m}}$ (resp. $\M^{\textnormal{nil}}, \M^{\textnormal{uni}}, \M^{\textnormal{et}}$) is a multiplicative (resp. nilpotent, unipotent, \'{e}tale) module. In fact, the second sequence is by taking Cartier dual of the first sequence, i.e., $\M^{\textnormal{uni}}=(\M^{\vee, \textnormal{nil}})^{\vee}$ and $\M^{\textnormal{et}}=(\M^{\vee, \textnormal{m}})^{\vee}$. Also, $\M^{\textnormal{m}}$ is the maximal multiplicative submodule of $\M$ (i.e., any multiplicative submodule of $\M$ is contained in $\M^{\textnormal{m}}$).

\item $\M$ is multiplicative (resp. nilpotent) if and only if $\M^{\vee}$ is \'{e}tale (resp. unipotent), and vice versa.
\end{enumerate}
\end{prop}
\begin{proof}
It is similar to Theorem 2.3.7 of \cite{Gao13}, and much easier. Let $\M^{\textnormal m}$ be as in Proposition \ref{prop max mult}. Then as in the proof of \textit{loc. cit.} shows, the quotient $\M^{\textnormal{nil}} :=\M /\M^{\textnormal m} =\oplus_{j=a+1}^d e_j$ is finite free. Define $\Fil^r \M^{\textnormal{nil}}:= \Fil^r \M / \Fil^r \M^{\textnormal m}$ (which injects into $\M/\M^{\rm m}$) with the induced $\phi_r$. $\M^{\textnormal{nil}}$ is clearly a nilpotent module in the category $\Mod_{k[u]/u^s}^{\phi}$. (2) is also easy to check.
\end{proof}

\begin{rem} \label{remark}
Similarly as in Remark 2.3.8 of \cite{Gao13}, suppose $\Fil^r \M =\oplus T_s\bolda$, and $\bolda = A \frac{\phi_r(\bolda)}{c^r}$, then $\M$ is unipotent if and only if $\Pi_{n=0}^{\infty}\phi^n(A)=0$.
\end{rem}

For $ep \geq t >s \geq p$, we can define a natural functor $\M_{t,s}$ from $\Modut$ to $\Modus$. For $\M \in \Modut$, let $\M_{t,s}(\M) := \M/u^s\M$, $\Fil^r \M_{t,s}(\M) := \Fil^r \M/u^s\M$ (which injects into $\M/u^s\M$), it is equipped with the induced $\phi_r$. Note that $\phi_r$ is well defined since $\phi_r(u^s\M)=\phi(u^{s-er})\phi_r(u^{er}\M)=u^{p(s-er)}\phi_r(u^{er}\M)=0$ because $p(s-er)\geq s$.

\begin{thm} \label{thm 4 list}
\begin{enumerate}

\item When $er=p-1$ and $ep \geq t >s >p$, $\M_{t,s}: \Mod_{k[u]/u^t}^{\phi} \to \Mod_{k[u]/u^s}^{\phi}$ is an equivalence.

\item When $er=p-1$ and $ep \geq t >p$, $\M_{t,p}:\Mod_{k[u]/u^t}^{\phi,u} \to \Mod_{k[u]/u^p}^{\phi,u}$ is an equivalence on the unipotent subcategories.

\item When $er \leq p-1$ and $ep \geq t >s \geq p$, $\M_{t,s}$ sends short exact sequences to short exact sequences.

\item When $er<p-1$ and $ep \geq t >s \geq p$, $\M_{t,s}: \Mod_{k[u]/u^t}^{\phi} \to \Mod_{k[u]/u^s}^{\phi}$ is an equivalence.
\end{enumerate}
\end{thm}

Before proving the theorem, we list several lemmas.
\begin{lemma} \label{lemma0}
\begin{enumerate}
\item Given $\M \in \Modus$, we can choose a basis $(e_1, \cdots, e_d)$ of $\M$ such that $\FM = \oplus T_s \bolda=\oplus_{i=1}^d \alpha_i =\oplus_{i=1}^d u^{x_i}e_i$. Let $\bolda=A( \frac{\phi_r(\bolda)}{c^r})$. Then there exists a matrix $B$ such that $AB=BA=u^{er}Id$.

\item Given $\M_1, \M_2 \in \Modus$, let $A_1, A_2$ be the matrices constructed as in (1), then a morphism $f: \M_1 \to \M_2$ is determined by some matrix $X$ such that $A_1\phi(X) = XA_2$. $X$ is not uniquely determined, but $\phi(X)$ is uniquely determined.
\end{enumerate}
\end{lemma}

\begin{proof}
For (1), the existence of a matrix $C$ (not unique) such that $CA=u^{er}Id$ is clear because $\FM \supseteq u^{er}\M$, but it does not guarantee $AC=u^{er}Id$ because $u$ is not an integral element in $k[u]/u^s$.
In our situation, $(u^{x_1}e_1, \cdots, u^{x_d}e_d)^T = A(\frac{\phi_r(\bolda)}{c^r})$. So we have $[u^{x_1}, \cdots, u^{x_d}](e_1, \cdots, e_d)^T=A\frac{\phi_r(\bolda)}{c^r}$. So $A=[u^{x_1}, \cdots, u^{x_d}]P$, where $P$ is the invertible matrix such that $(e_1, \cdots, e_d)^T=P\frac{\phi_r(\bolda)}{c^r}$. Thus we can take our $B=P^{-1}[u^{er-x_1}, \cdots, u^{er-x_d}]$.

For (2), if we make $\Fil^r \M_1=\oplus T_s\bolda, \Fil^r \M_2=\oplus T_s \boldb$ for some $\bolda$ and $\boldb$, then we can choose a matrix $X$ such that $f(\bolda)=X\boldb$.
\end{proof}

\begin{lemma}\label{lemma1}
Suppose $ep \geq t>s \geq p$. Let $A_1, A_2, B_1, B_2, X, Q$ be matrices with coefficients in $k[u]/u^t$ such that $B_1A_1 =u^{p-1}Id$ and $B_2A_2 =u^{p-1}Id$, and we have a relation $A_1\phi(X) -XA_2 = u^sQ$.
Then there exists a matrix $\hat{X}$ with coefficients in $k[u]/u^t$ such that $A_1 \phi(\hat{X})=\hat{X}A_2$ and $\phi(\hat X) \equiv \phi(X) (\bmod u^s)$, if either of the following conditions is satisfied:
\begin{enumerate}
\item $s>p$.
\item $s=p$ and $\prod_{i=0}^{\infty} \phi^i(A_1)= 0$.
\end{enumerate}
\end{lemma}
\begin{proof}
Let $$\hat{X}  =X+ u(u^{s-p}QB_2+   \sum_{n=0}^{\infty}( \prod_{i=0}^n \phi^i(A_1) \phi^{n+1}(u^{s-p}Q)\prod_{j=n+1}^{0}\phi^{j}(B_2) )). $$ Here $\prod_{j=n+1}^{0}\phi^{j}(B_2)$ means $\phi^{n+1}(B_2)\phi^{n}(B_2)\cdots \phi(B_2)B_2$.
\end{proof}

\begin{lemma}\label{lemma2}
Suppose $ep \geq t>s \geq p$. Let $A_1, A_2, B_1, B_2, X$ be matrices with coefficients in $k[u]/u^t$ such that $B_1A_1 =B_2A_2=u^{p-1}Id$, and we have a relation $A_1 \phi(X)=XA_2$. Suppose $\phi(X) \equiv 0 (\bmod u^s)$, then $\phi(X)=0$ if either of the following conditions is satisfied:
\begin{enumerate}
\item $s>p$.
\item $s=p$ and $\prod_{i=0}^{\infty} \phi^i(A_1)= 0$.
\end{enumerate}
\end{lemma}
\begin{proof}
For the second condition, since $\phi(X) \equiv 0 (\bmod u^p)$, we can let $X =uY$, then $A_1 u^p \phi(Y) =uYA_2$. Multiply $B_2$ on both sides, then we have $u^p(A_1\phi(Y)B_2 -Y)=0$, so $A_1\phi(Y)B_2 -Y = u^{t-p}Q$ for some matrix $Q$. We claim that the matrix equation $A_1\phi(Y)B_2 -Y = u^{t-p}Q$ with indeterminate $Y$ has a unique solution. Suppose we have two solutions $Y_1, Y_2$, and let $Z=Y_1 -Y_2$, then $Z=A_1\phi(Z)B_2$, so
$Z=A_1\phi(Z)B_2= A_1\phi(A_1)\phi^2(Z)\phi(B_2)B_2 = \cdots =0$ because $\prod_{i=0}^{\infty} \phi^i(A_1)= 0$. And the unique solution is
$$Y = u^{t-p}Q +  \sum_{n=0}^{\infty}  (\prod_{i=0}^{n} \phi^i(A_1) \phi^{n+1}(u^{t-p}Q)  \prod_{j=n}^{0} \phi^j(B_2) ).$$
So we have $Y=u^{t-p}W$. Now $X=u^{t-p+1}W$, so $u^{p(t-p+1)}$ divides $\phi(X)$. Since $p(t-p+1) \geq t$, so $u^t$ divides $\phi(X)$.

For the first condition, we can set $X=u^{\lceil \frac{s}{p} \rceil}Y$, and then it follows from a similar argument as above.
\end{proof}

\begin{proof}(Proof of Theorem \ref{thm 4 list})
We first prove essential surjectivity of both statement (1) and (2). Given $\M_s \in \Modus$ for $s \geq p$, suppose $\FM_s = \oplus_{i=1}^d T_s\alpha_i$, and $(\alphaseq)^T =A( \phi_r(\alpha_1), \cdots, \phi_r(\alpha_d) )^T =A (\eseq)^T$, then there exists $B$ such that $BA=u^{p-1}Id$. Take any lift $\hat{A}, \hat{B}$ of $A, B$ respectively with elements in $k[u]/u^t$, then $\hat{B}\hat{A} = u^{p-1}Id+u^s Q$ for some $Q$, so $[(Id+ u^{s-p+1}Q)^{-1}\hat{B}](\hat{A})=u^{p-1}$. Now define $\M_t = \oplus_{i=1}^d T_t \hat{e_i}$, $\Fil^r \M_t =\sum_{i=1}^d T_t (\hat{\alpha_i})$ with $(\hat{\alpha_1}, \cdots, \hat{\alpha_d})^T = \hat{A} (\hat{e_1}, \cdots, \hat{e_d})^T$, and $\phi_r(\hat{\alpha_i}) =\hat{e_i}$. Then clearly $\M_t$ is a preimage of $\M_s$.

Now we prove the full faithfulness of both statement (1) and (2). Let $\M_{t_1}, \M_{t_2} \in \Mod_{k[u]/u^t}$, and $\M_{s_1}= \M_{t,s}(\M_{t_1}), \M_{s_2}= \M_{t,s}(\M_{t_2})$. We need to prove that $h: \Hom(\M_{t_1}, \M_{t_2}) \to \Hom(\M_{s_1}, \M_{s_2})$ is bijective.
Let $A_1, A_2$ be the matrix for $\M_{t_1}, \M_{t_2}$ respectively as in Lemma \ref{lemma0}, then $A_{1,s} \equiv A_1 \bmod u^s, A_{2,s}\equiv A_2\bmod u^s$ are the matrix for $\M_{s_1}, \M_{s_2}$. To show surjectivity of $h$, given any morphism in $\Hom(\M_{s_1}, \M_{s_2})$ is equivalent to a matrix $X_s \in \Mat(k[u]/u^s)$ such that $A_{1, s}\phi(X_s) =X_sA_{2,s}$, now lift $X_s$ to any $X \in \Mat(k[u]/u^t)$, so $A_1\phi(X) -XA_2 = u^sQ$, and we can apply Lemma \ref{lemma1} to conclude surjectivity.
To show injectivity of $h$, suppose a morphism in $\Hom(\M_{t_1}, \M_{t_2})$ maps to $0$. This morphism is determined by some $X$ as in the conditions of Lemma \ref{lemma2}, so the morphism is itself $0$.

For (3), given a short exact sequence $0 \to \M_1 \to \M \to \M_2 \to 0$ in $\Mod_{k[u]/u^t}^{\phi}$, since these are finite free modules, it is clear that $0 \to \M_1/u^s\M_1 \to \M/u^s\M \to \M_2/u^s\M_2 \to 0$ is still short exact.
For the filtration sequence, $\Fil^r\M_1/u^s\M_1 \to \Fil^r\M/u^s\M \to \Fil^r\M_2/u^s\M_2$, the second map is clear surjective.
Since $u^s\M_1 \subseteq u^s\M \cap \Fil^r \M_1 \subseteq  u^s\M \cap \M_1 =u^s\M_1$, so $u^s\M \cap \Fil^r \M_1 =u^s \M_1$, and the first map is injective. Exactness in the center is also easily checked.

For (4), it can be similarly proved as (1) and (2) by modifying Lemma \ref{lemma1} and Lemma \ref{lemma2}(which in fact becomes easier).

\end{proof}

\begin{lemma} \label{unip}
Given a short exact sequence $0 \to \M_1 \to \M \to \M_2 \to 0$ in $\Modus$, then $\M$ is unipotent if and only if both $\M_1$ and $\M_2$ are unipotent.
\end{lemma}
\begin{proof}
Let $A_1, A_2$ be the matrix for $\M_1, \M_2$ respectively as in Lemma \ref{lemma0}. Then one can find some basis for $\M$ with the corresponding matrix $A=\begin{pmatrix}
A_1 & 0\\
\ast & A_2
\end{pmatrix}$. Then apply Remark \ref{remark}.
\end{proof}

\begin{thm} \label{thm abelian Kisin}
When $er \leq p-1$,
\begin{enumerate}
\item $\Mod_{k[u]/u^p}^{\phi}$ is an abelian category.
\item $\Mod_{k[u]/u^p}^{\phi,u}$ is an abelian subcategory.
\end{enumerate}
\end{thm}

\begin{proof}
For (1), when $er<p-1$, just as pointed out in Theorem 3.5.1 of \cite{Car06}, the proof is verbatim as that of Corollary 2.2.3.2 in \cite{Bre98b}. In fact, it also works for $er =p-1$.
But here, we give a more direct proof of this fact for all $er \leq p-1$ (without using the category $C_k$ in \cite{Bre98b}).

First, we show that if $f: \M_1 \to \M_2$ is a morphism in $\Mod_{k[u]/u^p}^{\phi}$, then $f(\M_1)$ is finite free.
We use notations from Lemma \ref{lemma0}, then $f(\M_1)$ is generated by $\phi(X)(f_1, \cdots, f_n)$, where $(f_1, \cdots, f_n)$ is a basis of $\M_2$, and $X$ is a matrix with coefficient in $k[u]/u^p$. Since $\phi(u)=u^p=0$  in $k[u]/u^p$, so $\phi(X)$ is in fact a matrix with coefficient in $k$. Thus we can easily show that $f(\M_1)$ is finite free, and so $(f(\M_1), f(\Fil^r \M_1), \phi_r)$ is an object in $\Mod_{k[u]/u^p}^{\phi}$.

Now, we show that if $f: \M_1 \to \M_2$ is a morphism in $\Mod_{k[u]/u^p}^{\phi}$, then $f(\Fil^r \M_1) =f(\M_1) \cap \Fil^r \M_2$. Since we have shown that $f(\M_1) \in \Mod_{k[u]/u^p}^{\phi}$, we can and do assume that $f: \M_1 \to \M_2$ is injective. By Lemma \ref{submodule}, we can assume that $\Fil^r \M_2$ is a direct sum of the form $\oplus_{i=1}^{n} T_p \alpha_i$, and $\Fil^r \M_1 =\oplus_{i=1}^{n} T_p u^{x_i}\alpha_i$. Suppose $x_i=0$ for $0 \leq i \leq a$, and $x_i >0$ for $a+1 \leq i \leq n$. Since $\phi_r(u^{x_i}\alpha_i)=\phi(u^{x_i})\phi_r(\alpha_i)=0$ for $x_i>0$, so $\M_1$ is generated by $\phi_r(\alpha_i)$ for $0 \leq i \leq a$. Thus $\M_1$ is of rank $a$, and we can choose a basis $(e_1, \cdots, e_a)$ of $\M_1$, such that $\Fil^r \M_1 =\oplus_{i=1}^a u^{y_j}T_pe_j =\oplus_{i=1}^n T_pu^{x_i}\alpha_i$. By Lemma \ref{submodule}(2), we must have $\Fil^r \M_1 =\oplus_{i=1}^{a} T_p \alpha_i$. Then we can easily deduce that $f(\Fil^r \M_1) =f(\M_1) \cap \Fil^r \M_2$.

Now, we show that if $f: \M_1 \to \M_2$ is a morphism in $\Mod_{k[u]/u^p}^{\phi}$, then the kernel $\mathcal K$ with $\Fil^r \mathcal K:= \mathcal K \cap \Fil^r \M_1$ and naturally induced $\phi_r$ is an object in $\Mod_{k[u]/u^p}^{\phi}$.
Since we have shown that $f(\M_1) \in \Mod_{k[u]/u^p}^{\phi}$, we can and do assume that $f: \M_1 \to \M_2$ is surjective and $f(\Fil^r \M_1)=\Fil^r \M_2$. Suppose that the rank of $\M_1$ and $\M_2$ is $d$ and $n$ respectively.
By Lemma \ref{submodule}(1), we can take $(e_1, \cdots, e_d)$ a basis of $\M_1$, such that $\mathcal K=\oplus_{i=1}^d T_p u^{x_i}e_i$ for some $0 \leq x_1 \leq \cdots \leq x_d \leq p$. Then $\M_2 = \oplus_{i=1}^d T_p \bar{e}_i$ where $\bar{e}_i$ is $u^{x_i}$-torsion. Since $\M_2$ is finite free, we conclude that $x_1=\cdots=x_{d-n}=0$ and $x_{d-n+1}=\cdots=x_d=p$. Thus $\mathcal K$ is finite free over $k[u]/u^p$ of rank $d-n$.
Now, clearly $\Fil^r \mathcal K \supseteq u^{er}\mathcal K$, and since $\phi_r$ and $f$ commute, $\phi_r(\Fil^r \mathcal K) \subseteq \mathcal K$. By Lemma \ref{submodule}, we can suppose $\Fil^r \mathcal K = \oplus_{i=1}^{d-n} \alpha_i, \Fil^r \M_2 =\oplus_{j=1}^n \beta_j$. For any $1\leq j \leq n$, take $\hat \beta_j \in \Fil^r \M_1$ such that $f(\hat \beta_j)=\beta_j$. Then $(\alpha_i, \hat \beta_j)_{i, j}$ generate $\Fil^r \M_1$, and so $(\phi_r(\alpha_i), \phi_r(\hat \beta_j))_{i, j}$ generate $\M_1$. Thus $(\phi_r(\alpha_i))_i$ generate $\mathcal K$.

Now we show that if $f: \M_1 \to \M_2$ is a morphism in $\Mod_{k[u]/u^p}^{\phi}$, then the cokernel $\mathcal N$ with naturally induced $\Fil^r$ and $\phi_r$ is an object in $\Mod_{k[u]/u^p}^{\phi}$.
Again, we can and do assume that $f$ is an injective morphism. Then by Lemma \ref{submodule}, we can suppose $\Fil^r \M_2=\oplus \alpha_i$, and $\Fil^r \M_1 =\oplus u^{x_i}\alpha_i$ for $x_i=0$ when $0 \leq i \leq a$, and $x_i>0$ when $i >a$. Since $\phi_r(u^{x_i}\alpha_i)=0$ if $x_i>0$, similarly as the end of the second paragraph of the proof shows, we must have $\Fil^r \M_1 =\oplus_{i=1}^a \alpha_i$. Then it is easy to deduce that $\Fil^r \mathcal N=\oplus_{i=a+1}^d\alpha_i$, $\mathcal N$ is finite free and is in fact generated by $\phi_r(\alpha_i), a+1 \leq i \leq d$.

For (2), use (1) and Lemma \ref{unip}.
\end{proof}

\begin{corollary} \label{Cor abelian Kisin}
\begin{enumerate}
\item When $er <p-1$, $\Mod_{k[u]/u^s}^{\phi}$ is an abelian category for $ep \geq s >p$.
\item When $er =p-1$, $\Mod_{k[u]/u^s}^{\phi, u}$ is an abelian category for $ep \geq s >p$.
\end{enumerate}
\end{corollary}

\begin{proof}
Combine Theorem \ref{thm 4 list} and Theorem \ref{thm abelian Kisin}. Note that we have given a new proof to Corollary 3.5.7 of \cite{Car06}.
\end{proof}

\section{Unipotent torsion Breuil modules}\label{section unip tor Breuil}

In this section, we prove that the category of unipotent torsion Breuil modules is an abelian category when $er=p-1, r<p-1$.

\newcommand{\MS}{\underline{\M}^{r}}
\newcommand{\MSu}{\underline{\M}^{r,u}}

Let $\MS$ be the category consisting of objects $(\M, \Fil^r \M, \phi_r, N)$ (called torsion Breuil modules) where
\begin{enumerate}
\item $\M= \oplus_{i \in I} S_{n_i}$ for a finite set $I$.
\item $\Fil^r \M$ is an $S$-submodule which contains $\Fil^rS\cdot \M$.
\item $\phi_r: \Fil^r \M \to \M$ is a Frobenius-semi-linear map such that $\phi_r(sx) =c^{-r}\varphi_r(s)\varphi_r(E(u)^rx)$ for $s\in \Fil^r S$
and $x \in \M$, and the image of $\phi_r$ generates $\M$.
\item $N: \M \to \M$ is a $W(k)$-linear map such that
\begin{itemize}
\item $N(sx)=N(s)x+sN(x)$ for all $s \in S, x \in \M$.
\item $E(u)N(\Fil^r \M) \subset \Fil^r \M$.
\item The following diagram is commutative:

$\begin{CD}
\Fil^r \M @>\phi_r>> \M \\
@V E(u)NVV @VVcNV\\
\Fil^r \M @>\phi_r>> \M.
\end{CD}$

\end{itemize}

\end{enumerate}

Morphisms in the category are $S$-linear maps that are compatible with $\Fil^r, \phi_r$ and $N$.

Let $\underline \M^{r, \phi}$ be the category similar to $\underline \M^{r}$ but without $N$, i.e., $\underline \M^{r, \phi}$ consists of objects $(\M, \Fil^r \M, \phi_r)$ satisfying (1), (2) and (3) above.

We denote the subcategory of $\MS$ consisting of objects killed by $p$ by $\Mod_{S_1}^{\phi, N}$, and the subcategory of $\underline \M^{r, \phi}$ consisting of objects killed by $p$ by $\Mod_{S_1}^{\phi}$.

When $r <p-1$, by the isomorphism $S_1/\Fil^p S_1 \simeq k[u]/u^{ep}$, there is a natural functor $\Mod_{S_1}^{\phi} \to  \Mod_{k[u]/u^{ep}}^{\phi}$ by sending $(\M, \Fil^r \M, \phi_r)$ to $(\M/\Fil^pS_1\M, \Fil^r \M/\Fil^pS_1\M, \phi_r)$. Note that this functor can be defined only when $r<p-1$, because we need to have $\phi_r(\Fil^pS_1\cdot \M)=0$.

Recall that in Definition 2.5.3 of \cite{Gao13}, for $\M \in \Mod_{S_1}^{\phi}$, if $\Fil^r \M =\oplus \bolda+ \Fil^p S_1\M$, $\frac{\phi_r(\bolda)}{c^r}=\bolde$, $\bolda = A\bolde$ with $A \in \Mat(S_1)$. Then $\M$ is called \textit{unipotent} (with respect to $\bolda$ and $\bolde$) if $\Pi_{n=1}^{\infty} \phi^n(A)=0$. The definition in \cite{Gao13} is only stated for $r=p-1$, but in fact it works for any $r \leq p-1$ (and any $e$), and it can be easily checked that the definition of unipotency is independent of choice of $\bolda$ and $\bolde$. Denote the unipotent subcategory by $\Mod_{S_1}^{\phi, u}$. An object $\M \in \Mod_{S_1}^{\phi, N}$ is called \textit{unipotent} if after forgetting $N$, it is a unipotent module in $\Mod_{S_1}^{\phi}$, we denote this unipotent subcategory by $\Mod_{S_1}^{\phi, N, u}$.

\begin{prop} \label{prop functor Kisin Breuil}
For any $e>0$, $r<p-1$, the functor $\Mod_{S_1}^{\phi} \to  \Mod_{k[u]/u^{ep}}^{\phi}$ is an
equivalence. It transforms short exact sequences to short exact sequences.
The functor also induces equivalence on the unipotent subcategories.
\end{prop}
\begin{proof}
The equivalence is Proposition 2.3.1 of \cite{Car06}.
To check the equivalence on unipotent subcategories, we only need to check that the functor and its inverse sends unipotent objects to unipotent objects, and we can use Remark \ref{remark} for this.
To check the exactness, let $0 \to \M_1 \to \M \to \M_2 \to 0$ be a short exact sequence in $\Mod_{S_1}^{\phi}$, to check short exactness of the resulting sequence, it suffices to check that $\Fil^r \M_1/\Fil^p S_1\M_1 \to \Fil^r\M/\Fil^pS_1\M$ is injective, i.e., $\Fil^r\M_1\cap \Fil^p S_1\M=\Fil^pS_1\M_1$. This is true because
$$\Fil^pS_1\M_1 \subseteq  \Fil^r\M_1\cap \Fil^p S_1\M  \subseteq \M_1 \cap\Fil^p S_1\M =\Fil^pS_1\M_1.$$
\end{proof}

\begin{thm}
\begin{enumerate}
\item When $er=p-1, r<p-1$, $\Mod_{S_1}^{\phi, u}$ is an abelian catetory.
\item When $er=p-1, r<p-1$, $\Mod_{S_1}^{\phi, N, u}$ is an abelian catetory.
\end{enumerate}
\end{thm}
\begin{proof}
For statement(1), combine Corollary \ref{Cor abelian Kisin}(2) and Proposition \ref{prop functor Kisin Breuil}. Statement (2) is easy corollary of (1) by keeping track of the $N$-action. For example, given a morphism $f: \M_1 \to \M_2$ in $\Mod_{S_1}^{\phi, N, u}$, then $\Ker f \in \Mod_{S_1}^{\phi, u}$ by (1). But $\Ker f$ also has the naturally induced monodromy operator $N$ (induced from that of $\M_1$, which makes it an object in $\Mod_{S_1}^{\phi, N, u}$. Similar argument works for $\Coker f$.
\end{proof}

For a module $\M \in \MS$, define $\Fil^r p^m \M:= p^m \Fil^r \M$, then $(p^m \M, \Fil^r p^m \M)$ with the induced $\phi_r$ and $N$ is an object in $\MS$.
For a module $\M \in \MS$, we also define $\M^{(p^m)}: =\{ x \in \M, p^mx=0\}$, and let $\Fil^r \M^{(p^m)}:= \M^{(p^m)} \cap \Fil^r \M$.

\begin{lemma}
When $r<p-1$, $\M^{(p)} \in \MS$.
\end{lemma}
\begin{proof}
The proof follows the same idea as in Lemma 2.3.1.2 of \cite{Bre98b}, except that we have to tensor with $S$ instead of $W(k)$.
We prove the lemma by an induction on the minimal $p$-power that kills $\M$. First suppose $p^2 \M=0$. Then we have the following commutative diagram:

$$\begin{CD}
@.  S\otimes_{\phi, S} \Fil^r \M^{(p)}  @>>> S\otimes_{\phi, S} \Fil^r \M  @>\times p>>  S\otimes_{\phi, S} p\Fil^r \M  @>>>0 \\
@.   @VV1\otimes \phi_r V @VV1\otimes \phi_rV @VV1\otimes \phi_rV\\
0@>>> \M^{(p)} @>>> \M @>\times p>>  p\M @>>> 0,
\end{CD}$$
where the top row is right exact and the bottom row is short exact. We only need to show that the vertical arrow on the left is surjective.

Given $x \in \M^{(p)}$, suppose $1\otimes\phi_r(\hat{x})=x$ for $\hat{x} \in S\otimes_{\phi, S} \Fil^r \M$, so $1\otimes\phi_r(p\hat{x})=0$.
Since $p\M \in \Mod_{S_1}^{\phi}$, by Lemma 2.2.1 of \cite{Car08} (note that we need $r<p-1$ here), we must have that $p\hat{x} \in  S\otimes_{\phi, S} (E(u)p\Fil^r\M +\Fil^pS \cdot p\M)$.
So $p\hat x= \sum_{i=1}^n a_i\otimes  p(E(u)\alpha_i+s_i\beta_i)$, where $a_i \in S, \alpha_i \in \Fil^r \M, s_i\in \Fil^p S, \beta_i \in \M$. Let $\hat y= \sum_{i=1}^n a_i\otimes  (E(u)\alpha_i+s_i\beta_i)$, then $p\hat x=p\hat y$, and $1\otimes \phi_r(\hat y)=pz$ for $z \in \M$ (again we need $r<p-1$ here, so that $\phi_r(\Fil^p S \M)=0$). Now suppose $1\otimes \phi_r(\hat z)=z$ for $\hat z \in  S\otimes_{\phi, S} \Fil^r \M $, then $\hat x -\hat y +p\hat z \in S\otimes_{\phi, S} \Fil^r \M^{(p)}$ and maps to $x$. And we are done.

Now suppose the lemma is true for $\M$ such that $p^m \M=0$. Then for $\M$ such that $p^{m+1}\M=0$ and $p^m \M \neq 0$, a similar process as above shows that $\M^{(p^m)} \in \MS$. By induction hypothesis $\M^{(p)}=(\M^{(p^m)})^{(p)} \in \MS$.

\end{proof}

Now, we define ``unipotency" for a module in $\MS$. We define it inductively.

\begin{defn} \label{unipotentdefn}
A module $\M \in \MS$ such that $p^2\M=0$ is called unipotent if $\M^{(p)}$ and $p\M$ are unipotent (as modules in $\Mod_{S_1}^{\phi}$). Inductively, $\M \in \MS$ such that $p^{m+1}\M=0$ and $p^m\M \neq 0$ is called unipotent if $\M^{(p)}$ and $p\M$ are unipotent. That is, $\M^{(p)}, (p\M)^{(p)}, \cdots, (p^{m-1}\M)^{(p)}$ and $p^m \M$ are all unipotent modules in $\Mod_{S_1}^{\phi}$.
We denote the unipotent subcategory by $\MSu$.
\end{defn}

\begin{lemma} Suppose $er=p-1, r<p-1$.
For $\M \in \MSu$, we have $p^i \Fil^r \M = p^i \M \cap \Fil^r \M$ for all $i$.

\end{lemma}

\begin{proof}
We prove by an induction on the minimal $p$-power that kills $\M$. Suppose $p^2 \M=0$ and $p \M \neq 0$. Then there is an injective morphism $p\M \to \M^{(p)}$. Since $\Mod_{S_1}^{\phi, N, u}$ is abelian, we have $p \Fil^r\M =p\M \cap \Fil^r\M^{(p)}=p\M \cap \Fil^r \M \cap \M^{(p)}=p\M \cap \Fil^r \M$.
Suppose the lemma is true for $\M$ such that $p^m \M=0$. Now suppose $p^{m+1}\M=0$ and $p^m \M \neq 0$. Then the injective morphism $p^m \M \to \M^{(p)}$ in the abelian category $\Mod_{S_1}^{\phi, N, u}$ gives us $p^m \Fil^r\M=p^m \M \cap \Fil^r \M^{(p)}=p^m \M \cap \Fil^r \M$.
We claim that the module $(\M/p^m\M, \Fil^r \M/p^m\Fil^r \M, \phi_r)$ is in $\MSu$. For any $0 \leq i \leq m-1$, we have the short exact sequence $0 \to (p^i\M)^{(p)}/p^m\M \to (p^i\M/p^m\M)^{(p)} \stackrel{h}{\longrightarrow} p^m\M \to 0$, where for $\bar{x} \in (p^i\M/p^m\M)^{(p)}$, take any lift $x \in p^i \M$, and define $h(\bar{x})=px$. Since $(p^i\M)^{(p)}/p^m\M$ and $p^m\M$ are unipotent, $(p^i\M/p^m\M)^{(p)}$ is unipotent by Lemma \ref{unip} (and Proposition \ref{prop functor Kisin Breuil}). And note that $(p^{m-1}\M/p^m\M)^{(p)}$ is just $p^{m-1}\M/p^m\M$.
So we have proved that $\M/p^m\M \in \MSu$. By the induction hypothesis, $p^i \Fil^r \M/p^m\Fil^r \M= (p^i \M/p^m\M) \cap (\Fil^r \M/p^m\Fil^r \M)$ for all $i$. It is then easy to deduce that $p^i \Fil^r \M = p^i \M \cap \Fil^r \M$ for all $i$.

\end{proof}

\begin{thm} \label{thm unip Breuil}
When $er=p-1$ and $r<p-1$, $\MSu$ is an abelian category
\end{thm}

The proof of the theorem follows the same strategy as in Section 2.3 of \cite{Bre98b}. We have already shown that the subcategory with objects killed by $p$ is abelian, so we only need to do a d\'evissage argument as in Section 2.3 of \cite{Bre98b}.

The following lemma will be useful.
\begin{lemma}\label{modp}
For $\M \in \MSu$, then $\M/p\M \in \MSu$.
\end{lemma}

\begin{proof}
We prove by an induction on the minimal $p$-power that kills $\M$.
When $p^2 \M=0$, then the short exact sequence $0 \to \M^{(p)}/p\M \to \M/p\M \stackrel{\times p}{\longrightarrow} p\M \to 0$ concludes the result.
Suppose the lemma is true for $\M$ such that $p^m \M=0$.
Now suppose $p^{m+1}\M=0$ and $p^m \M \neq 0$. Since $p\M$ satisfies the induction hypothesis, $p\M/p^2\M \in \MSu$. Then the short exact sequence $0 \to \M^{(p)}/(p\M)^{(p)} \to \M/p\M  \stackrel{\times p}{\longrightarrow} p\M/p^2\M \to 0$ shows that $\M/p\M$ is unipotent.
\end{proof}

We begin with generalizing Lemma 2.3.1.3 of \cite{Bre98b} to the $er=p-1, r<p-1$ unipotent situation.

\newcommand{\K}{\mathcal K}

\begin{lemma} \label{lemma kernel}
Let $f: \mathcal N \to \M$ be a morphism in $\MSu$ which is surjective on the $S$-modules. Suppose that $p\M=0$. Then $\Fil^r \mathcal N \to \Fil^r \M$ is surjective and $(\K,\Fil^r \K, \phi_r)$ is a kernel of $f$ in $\MSu$, here $\K := \ker f, \Fil^r \K := \K \cap \Fil^r \mathcal N$.
\end{lemma}

\begin{proof}
The steps of the proof are exactly the same as Lemma 2.3.1.3 of \cite{Bre98b}, we just need to check the ``unipotency" on each step. For the convenience of the reader, we give a sketch of the proof here.
If $p\mathcal N=0$, then the lemma is true because $\Mod_{S_1}^{\phi, N, u}$ is abelian.
We do an induction on the minimal $p$-power that kills $\mathcal N$. Suppose the lemma is true for $\mathcal N$ such that $p^{m-1}\mathcal N=0$, and suppose now $p^m \mathcal N=0, p^{m-1}\mathcal N \neq 0$.
By the same proof as in Lemma 2.3.1.3 of \cite{Bre98b}, we have $\Fil^r \K \cap p\K=p\Fil^r \K$.
The exact sequence $0 \to \K\cap \mathcal N^{(p)} \to \mathcal N^{(p)} \to \M$ in $\Mod_{S_1}^{\phi, N}$ shows that $\K\cap \mathcal N^{(p)} (=\K^{(p)})$ and $\mathcal N^{(p)}/\K\cap \mathcal N^{(p)}$ are unipotent.
The short exact sequence $0 \to \mathcal N^{(p)}/\K\cap N^{(p)} \to \mathcal \mathcal N/\mathcal K \stackrel{\times p}{\longrightarrow} p\mathcal N/p\K \to 0$ shows that $p\mathcal N/p\K$ is unipotent.
Now $p\K = \ker(p\mathcal N \to p\mathcal N/p\K)$, since $p^{m-1}(p\mathcal N)=0$, so by the induction hypothesis, $p\K \in \MSu$.
Then, the short exact sequence $0 \to \mathcal K/p\mathcal N \to \mathcal N/p\mathcal N \to \M \to 0$ ($\mathcal N/p\mathcal N$ is unipotent by Lemma \ref{modp}) shows that $\mathcal K/p\mathcal N$ is unipotent.
And the short exact sequence $0 \to p\mathcal N/p\mathcal K \to \mathcal K/p\mathcal K \to \mathcal K/p\mathcal N \to 0$ of $S_1$-modules shows that $\K/p\K$ is finite free over $S_1$. By Lemma 2.3.1.1 of \cite{Bre98b}, $\mathcal K \in \MS$. And we are done.
\end{proof}

\begin{proof}[Proof of Theorem \ref{thm unip Breuil}]
Firstly we claim that if $f: \mathcal N \to \M$ is a morphism in $\MSu$, then $(\K,\Fil^r \K, \phi_r)$ is a kernel of $f$ in $\MSu$. But this is just the $er=p-1, r<p-1$ unipotent generalization of Proposition 2.3.2.1 of \cite{Bre98b}. Just apply Lemma \ref{lemma kernel}, and the proof is almost verbatim.

Secondly we claim that if $f: \mathcal N \to \M$ is a morphism in $\MSu$, then $f(\Fil^r \mathcal N)=\Fil^r \M \cap f(\mathcal N)$, and the naturally defined cokernel $(\mathcal C, \Fil^r \mathcal C, \phi_r)$ is a cokernel of $f$ in $\MSu$. And this is a generalization of Proposition 2.3.2.2 of \cite{Bre98b} to the $er=p-1, r<p-1$ unipotent situation. The reader can easily check that everything works through if we check ``unipotency" at each step.
\end{proof}

\begin{prop} \label{prop unip in the middle}
Suppose $er=p-1$ and $r<p-1$. Let $0 \to \M_1 \to \M \to \M_2 \to 0$ be a short exact sequence in $\MS$ where $\M_1, \M_2$ are unipotent. Then $\M$ is also unipotent.
\end{prop}

\begin{proof}
Apply snake lemma to the following commutative diagram,

$$\begin{CD}
0@>>> \M_1 @>>> \M @>>>  \M_2 @>>> 0 \\
@.   @VV\times p V @VV\times p V @VV\times p V\\
0@>>> \M_1 @>>> \M @>>>  \M_2 @>>> 0.
\end{CD}$$

We have the exact sequence $0 \to \M_1^{(p)} \to \M^{(p)} \to \M_2^{(p)} \stackrel{h}{\longrightarrow} \M_1/p\M_1$, where $h$ is the connecting homomorphism. Then it is easy to conclude that $\M^{(p)}\in \MSu$.
We now prove that $p\M$ is unipotent. We prove this by an induction on the minimal $p$-power that kills $\M_1$. If $p\M_1=0$, then by the short exact sequence $0 \to \M^{(p)}/\M_1 \to \M/\M_1 (=\M_2) \to \M/\M^{(p)} \to 0$, and since $\MSu$ is abelian, $p\M =\M/\M^{(p)} \in \MSu$.
Suppose the proposition is true for $\M_1$ such that $p^m \M_1 =0$. Now suppose $p^{m+1}\M_1 =0$ and $p^m \M_1 \neq 0$. The short exact sequence $0 \to \M^{(p)}/\M_1^{(p)} \to \M/\M_1 \stackrel{\times p}{\longrightarrow} p\M/p\M_1 \to 0$ shows that $p\M/p\M_1 \in \MSu$. Then apply induction hypothesis to the short exact sequence $0 \to p\M_1 \to p\M \to p\M/p\M_1 \to 0$ to conclude that $p\M$ is unipotent.

\end{proof}

\begin{corollary}
Let $\mathcal C$ be the smallest full subcategory of $\MS$ which contains $\Mod_{S_1}^{\phi, N, u}$ and is stable by extension. Then $\mathcal C$ is the same as $\MSu$.
\end{corollary}
\begin{proof}
By Proposition \ref{prop unip in the middle}, $\mathcal C$ is a subcategory of $\MSu$. Now for any object $\M \in \MSu$, by an easy induction process on the minimal $p$-power that kills $\M$, and using the short exact sequence $0 \to p\M \to \M \to \M/p\M \to 0$ (i.e., $\M$ is an extension of $p\M$ and $\M/p\M$), it is clear that $\M \in \mathcal C$.
\end{proof}

\bibliographystyle{alpha}

\end{document}